\documentclass[12pt, dvips, twoside]{amsart}
\usepackage{amsmath,amssymb,amscd,amsxtra,latexsym,bbm,graphicx,epsfig,epic,
eepic,oldgerm,bm,euscript,psfrag,amsthm}
\usepackage{a4wide}
\usepackage[dvips]{color}

\newtheorem{Theorem}{Theorem}[section]      
\newtheorem{Proposition}[Theorem]{Proposition}    
\newtheorem{Lemma}[Theorem]{Lemma}            

\theoremstyle{definition}  
\newtheorem{Remark}{Remark}

\newcounter{mnotecount}[section]

\newcommand{\rmnote}[1]{}

\title[Extremal domains for the first eigenvalue]{Extremal domains for the first eigenvalue in a general compact Riemannian manifold}

\author[E. Delay]{Erwann Delay}
\address{Erwann Delay,
Laboratoire d'analyse non lin\'eaire et g\'eom\'etrie, Facult\'e
des Sciences, 33 rue Louis Pasteur, 84000 Avignon, France}
\email{Erwann.Delay@univ-avignon.fr}
\urladdr{http://www.univ-avignon.fr/en/research/annuaire-chercheurs/\newline
$\mbox{ }$\hfill membrestruc/personnel/delay-erwann-1.html}

\author[P. Sicbaldi]{Pieralberto Sicbaldi}
\address{Pieralberto Sicbaldi, Universit\'e d'Aix-Marseille, Laboratoire d'Analyse Topologie Probabilit\'e, 39 rue Joliot-Curie, 13453 Marseille Cedex 13, France}
\email{pieralberto.sicbaldi@univ-cezanne.fr}
\urladdr{http://www.latp.univ-mrs.fr/~sicbaldi}
\begin{document}

\maketitle

{\bf Abstract.} We prove the existence of extremal domains with small prescribed volume for the first eigenvalue of the Laplace-Beltrami operator in any compact Riemannian manifold. This result generalizes a results of F. Pacard and the second author where the existence of a nondegenerate critical point of the scalar curvature of the Riemannian manifold was required. 

\tableofcontents

\section{Introduction and statement of the result}

Let $(M,g)$ be an $n$-dimensional Riemannian manifold, $\Omega$ a connected and open domain in $M$ with smooth boundary, and $\lambda_\Omega>0$ the first eigenvalue of the Laplace-Beltrami operator $-\Delta_{g}$ in $\Omega$ with $0$ Dirichlet boundary condition. The domain $\Omega$ is said to be {\it extremal} (for the first eigenvalue of the Laplace-Beltrami operator under 0 Dirichlet boundary condition) if it is a critical point for the functional $\Omega \longmapsto \lambda_\Omega$ in the class of domains with the same volume. 

\medskip

An extremal domain is characterized by the fact that the first eigenfunction of the Laplace-Beltrami operator with 0 Dirichlet boundary condition has constant Neumann data at the boundary. This result has been proved in the Euclidean space by P.R. Garabedian and M. Schiffer in 1953 \cite{Gar-Schif}, and in a general Riemannian manifold by A. El Soufi and S. Ilias in 2007 \cite{ElS-Ilias}. Extremal domains are then domains where the elliptic overdetermined problem
\begin{equation}\label{systeme}
\left\{
\begin{array}{rclll}
	\displaystyle \Delta_g u + \lambda  \, u & = & 0 & \textnormal{in} & \Omega \\
	\displaystyle u & > & 0 & \textnormal{in} & \Omega \\
	\displaystyle u & = & 0 & \textnormal{on} & \partial \Omega \\
\displaystyle g(\nabla u, \nu) & = & \textnormal{constant} & \textnormal{on} & \partial\Omega  \, 
\end{array}
\right.
\end{equation}
can be solved for some positive constant $\lambda$, where $\nu$ denotes the outward unit normal vector about $\partial \Omega$ for the metric $g$.

\medskip

In $\mathbb{R}^n$ the only extremal domains are balls. This is a consequence of a very well known result of J. Serrin: if there exists a solution $u$ to the overdetermined elliptic problem 
\begin{equation}\label{serrin}
\left\{
\begin{array}{rclll}
	\displaystyle \Delta u + f(u) & = & 0 & \textnormal{in} & \Omega \\
	\displaystyle u & > & 0 & \textnormal{in} & \Omega \\
	\displaystyle u & = & 0 & \textnormal{on} & \partial \Omega \\
\displaystyle \langle \nabla u, \nu \rangle & = & \textnormal{constant} & \textnormal{on} & \partial\Omega  \, ,
\end{array}
\right.
\end{equation}
for a given bounded domain $\Omega \subset \mathbb{R}^n$ and a given Lipschitz function $f$, where $\nu$ denotes the outward unit normal vector about $\partial \Omega$ and $\langle \cdot, \cdot \rangle$ the scalar product in $\mathbb{R}^n$, then $\Omega$ must be a ball, \cite{Serrin}. In the Euclidean space, round balls are in fact not only extremal domains, but also minimizers for the first eigenvalue of the Laplacian with 0 Dirichlet boundary condition in the class of domains with the same volume. This follows from the Faber--Kr\"ahn inequality, 
\begin{equation} \label{e:faber-krahn}
\lambda_\Omega \geq \lambda_{B^n(\Omega)} 
\end{equation}
where $B^n(\Omega)$ is a ball of $\mathbb{R}^n$ with the same volume 
as $\Omega$, because equality holds in~\eqref{e:faber-krahn} if and only if 
$\Omega = B^n(\Omega)$, see~\cite{Faber} and~\cite{Krahn}. 

\medskip

Nevertheless, very few results are known about extremal domains in a Riemannian manifold. The result of J. Serrin, based on the moving plane argument introduced by A. D. Alexandrof in \cite{Alex}, uses strongly the symmetry of the Euclidean space, and naturally it fails in other geometries.
The classification of extremal domains is then achieved in the Euclidean space, but it is completely open in a general Riemannian manifold.

\medskip

For small volumes, a method to build new examples of extremal domains in some Riemannian manifolds has been developed in \cite{Pac-Sic} by F. Pacard and P. Sicbaldi. They proved that when the Riemannian manifold has a nondegenerate critical point of the scalar curvature, then it is possible to build extremal domains of any given volume small enough, and such domains are close to geodesic balls centered at the nondegenerate critical point of the scalar curvature. The method fails if the Riemannian method does not have a nondegenerate critical point of the scalar curvature.
\medskip

In this paper we improve the result of F. Pacard and P. Sicbaldi by eliminating the hypothesis of the existence of a nondegenerate critical point for the scalar curvature. In particular, we are able to build extremal domains of small volume in every compact Riemannian manifold.

\medskip

For $\epsilon >0$, we denote by $B^g_{\epsilon}(p) \subset M$ the geodesic ball of center $p \in M$ and radius $\epsilon$. We denote by  $B_\epsilon \subset \mathbb R^n$ the Euclidean ball of radius $\epsilon$ centered at the origin. The main result of the paper is the following:
\begin{Theorem}\label{maintheorem} \label{th:1}
Let $M$ be a Riemannian manifold of dimension $n \geq 2$. There exist $\epsilon_0 > 0$ and a smooth function 
$$
\Phi:M\times(0,\epsilon_0)\longrightarrow \mathbb{R}
$$
such that:
\begin{itemize}
\item[(1)] For all $\epsilon\in (0,\epsilon_0)$, if $p$ is a critical point of the function $\Phi(\cdot,\epsilon)$
then there exists an extremal domain $\Omega_\epsilon \subset M$ whose volume is equal to the Euclidean volume of $B_\epsilon$. Moreover, there exists $c >0$ and, for all $\epsilon \in (0, \epsilon_0)$, the boundary of $\Omega_\epsilon$ is a normal graph over $\partial B^g_\epsilon (p)$ for some function $v(p,\epsilon)$ with
\[
\|Êv(p,\epsilon) \|_{\mathcal C^{2, \alpha}(\partial B^g_\epsilon (p))} \leq c \, \epsilon^3 \, .
\]
\item[(2)] There exists a function $\bf r$ defined on M that can be written as
\[
{\bf r} = K_1\, \|Riem\|^2+K_2\, \|Ric\|^2+K_3\, R^2+K_4\, \Delta_g R
\]
where $Riem$, $Ric$, $R$ denote respectively the Riemann curvature tensor, the Ricci curvature tensor and the scalar curvature of $(M,g)$, and $K_1, K_2, K_3$ and $K_4$ are constants depending only on $n$, such that for all $k\geq0$ 
$$
\|\Phi(p,\epsilon)-R_p-\epsilon^2\, {\bf r}_p\|_{C^k(M)}\leq c_k\, \epsilon^3 \, 
$$
for some constant $c_k>0$ which does not depend on $\epsilon\in(0,\epsilon_0)$ (the subscript $p$ means that we evaluate the function at $p$).
\item[(3)] The following expansion holds:
$$
\begin{array}{lll}
\displaystyle \lambda_{\Omega_\epsilon} & = & \displaystyle \lambda_1\, \epsilon^{-2}-\frac{n(n+2)+2\lambda_1}{6n(n+2)}\, \Phi(p,\epsilon)\\[4mm]
& = & \displaystyle \lambda_1\, \epsilon^{-2} 
-\frac{n(n+2)+2\lambda_1}{6n(n+2)}\, \left(R_p+\epsilon^2\, {\bf r}_p\right)+\mathcal{O}(\epsilon^3) 
\end{array}
$$
where $\lambda_1$ is the first Dirichlet eigenvalue of the unit Euclidean ball.
\end{itemize}
\end{Theorem}

\medskip 

The explicit computation of the constants $K_i$ is given in section \ref{localisationD}. We remark that if $M$ is compact, then there exists always a critical point of $\Phi(\cdot, \epsilon)$, and then we have small extremal domains obtained as perturbation of small geodesic balls in every compact Riemannian manifold without boundary. 

\medskip

It is clear that this theorem generalizes the result of \cite{Pac-Sic} because the construction of extremal domains does not require the existence of a nondegenerate critical
point of the scalar curvature. In fact, if the scalar curvature function $R$ has a nondegenerate critical point $p_0$, then for all $\epsilon$ small enough there exists a critical point $p = p(\epsilon)$ of $\Phi(\cdot,\epsilon)$ such that
\[
\textnormal{dist}(p, p_0) \leq c\, \epsilon^2.
\]
and then the geodesic ball $B^g_\epsilon(p)$ can be perturbed in order to obtain an extremal domain. We recover in this case the result of \cite{Pac-Sic}, but with a better estimation of the distance of $p$ to $p_0$ (in \cite{Pac-Sic} the distance between $p$ and $p_0$ is bounded by $c\, \epsilon$). In particular, we have the $p$-{\it independent} expansion 
$$
\lambda_{\Omega_\epsilon}=\lambda_1\, \epsilon^{-2}-\frac{n(n+2)+2\lambda_1}{6n(n+2)}\, R_{p_0}+\mathcal{O}(\epsilon^2)
$$
The result of \cite{Pac-Sic} can not be applied to some natural metrics such that a Einstein one, i.e when $Ric = k\, g$ for some constant $k$, or simply a constant scalar curvature one. 
In the case where $R$ is a constant function, one gets the existence of extremal domains close to any nondegenerate critical point of the function ${\bf r}$. In the particular case where the metric $g$ is Einstein we obtain extremal domains close to any nondegenerate critical point of the function (we will see that $K_1\neq0$)
\[
p \to \|Riem_p\|^2\, .
\]

\medskip

In order to put the result in perspective let us digress slightly. The solutions of the isoperimetric problem
\[
I_\kappa : = \min_{\Omega \subset M \, : \, {\rm Vol}_g \, \Omega = \kappa} \mbox{Vol}_{g_\textnormal{in}} \, \partial \Omega 
\]
are (where they are smooth enough) constant mean curvature hypersurfaces (here $g_{\textnormal{in}}$ denotes the induced metric on the boundary of $\Omega$). In fact, constant mean curvature are the critical points of the area functional 
\[
\Omega \to \mbox{Vol}_{g_\textnormal{in}} \, \partial \Omega 
\]
under a volume constraint ${\rm Vol}_g \, \Omega = \kappa$. Now, it is well known (see \cite{Faber}, \cite{Krahn} and \cite{Kra-2}) that the determination of the isoperimetric profile $I_\kappa$ is related to the Faber-Kr\"ahn profile, where one looks for the least value of the first eigenvalue of the Laplace-Beltrami operator amongst domains with prescribed volume
\[
FK_\kappa : = \min_{\Omega \subset M \, : \, {\rm Vol}_g \, \Omega = \kappa } \lambda_{\Omega} 
\]
A smooth solution to this minimizing problem is an extremal domain, and in fact extremal domains are the critical points of the functional 
\[
\Omega \to \lambda_{\Omega}
\]
under a volume constraint ${\rm Vol}_g \, \Omega = \kappa$. 

\medskip

The result of F. Pacard and P. Sicbaldi \cite{Pac-Sic} had been inspired by some parallel results on the existence of constant mean curvature hypersurfaces in a Riemannian manifold $M$. In fact, R. Ye built in \cite{Ye} constant mean curvature topological spheres which are close to geodesic spheres of small radius centered at a nondegenerate critical point of the scalar curvature, and the result of F. Pacard and P. Sicbaldi can be considered the parallel of the result of R. Ye in the context of extremal domains. The method used in \cite{Pac-Sic} is based on the study of the operator that to a domain associates the Neumann value of its first eigenfunction, which is a nonlocal first order elliptic operator. This represents a big difference with respect to the result of R. Ye, where the operator to study was a local second order elliptic operator.

\medskip

In a recent paper, \cite{Pac-Xu}, F. Pacard and X. Xu generalize the result of R. Ye by eliminating the hypothesis of the existence of a nondegenerate critical point of the scalar curvature function. For every $\epsilon$ small enough, they are able to build a small topological sphere of constant mean curvature equal to $\frac{n-1}{\epsilon}$ by perturbing a small geodesic ball centered at a critical point of a certain function defined on $M$ which is close to the scalar curvature function. For this, they use the variational characterization of constant $H_0$ mean curvature hypersurfaces as critical points of the functional 
\[
S \to {\rm Vol}_{g_\textnormal{in}}(S) - H_0\,  {\rm Vol}_{g}(D_S)
\]
in the class of topological sphere, where $D_S$ is the domain enclosed by $S$, see \cite{Pac-Xu}.

\medskip 

Our construction is based on some ideas of \cite{Pac-Xu}. For this, we use the variational characterization of extremal domains. The main difference and difficulties with respect to the result of F. Pacard and X. Xu arise in the fact that there does not exist an explicit formulation to compute the first eigenvalue of a domain while there exists an explicit formulation to compute the volume of a surface. 

\medskip

Our result shows once more the similarity between constant mean curvature hypersurfaces and extremal domains. The deep link between such two objects has been underlined also in \cite{Ros-Sic} and \cite{Sch-Sic}. 

\medskip

It is important to remark that P. Sicbaldi was able to build extremal domains of big volume in some compact Riemannian manifold without boundary by perturbing the complement of a small geodesic ball centered at a nondegenerate critical point of the scalar curvature function, see \cite{Sic}. As in the case of small volume domains, the existence of a nondegenerate critical point of the scalar curvature function is mandatory (and such result requires also that the dimension of the manifold is at least 4). It would be interesting to adapt our result in order to build extremal domains of big volume in any compact Riemannian manifold without boundary by perturbing the complement of small geodesic balls of radius $\epsilon$ centered at a critical point of the function $\Phi(\cdot, \epsilon)$ or some other similar function.   This result would allow for example to obtain extremal domains $\Omega_\epsilon$ that are given by the complement of a small topological ball in a flat 2-dimensional torus, and by the characterization of extremal domains this would lead to a nontrivial solution of (\ref{serrin}), with $f(t) = \lambda\, t$, in the universal covering $\tilde \Omega_\epsilon$ of $\Omega_\epsilon$, which is a nontrivial unbounded domain of $\mathbb{R}^2$. Up to our knowledge the existence of this unbounded domain is not known. Remark that $\tilde \Omega_\epsilon$ is a double periodic domain, made by the complement of a infinitely countable union of topological balls. The existence of $\tilde \Omega_\epsilon$ would establish once more the strong link between extremal domains and constant mean curvature surfaces, via the double periodic constant mean curvature surfaces (see \cite{GBr}, \cite{Rit} and \cite{Rit-th}).

\medskip

{\bf Acknowledgement.} Both authors are grateful to Philippe Delano\"e for his pleasant ``S\'eminaire commun d'analyse g\'eom\'etrique" that took place at CIRM (Marseille) in september 2012, where they met
and started the collaboration. This work was done from september 2012 to january 2013, when the first author was member of the Laboratoire d'Analyse Topologie et Probabilit\'e of the Aix-Marseille University as ``chercheur CNRS en d\'el\'egation", and he his grateful to the member of such research laboratory for their warm hospitality.  The first author is partially supported by 
the ANR-10-BLAN 0105  ACG and the ANR SIMI-1-003-01.

\section{Notations and preliminaries}

Let $\Omega_0$ be a smooth bounded domain in $M$. We say that $\{ \Omega_t \}_{t \in (-t_0, t_0)}$ is a deformation of $\Omega_0$ if there exists a vector field $\Xi$ such that 
$\Omega_t = \xi (t, \Omega_0)$ where  $\xi(t, \cdot)$ is the flow associated to $\Xi$, namely 
\[
\frac{d\xi}{dt} (t,p)= \Xi (\xi(t,p)) \qquad \mbox{and} \qquad \xi(0, p) =p \,.
\] 
In this case we say that $\Xi$ is the vector field that generates the deformation.
The deformation is said to be volume preserving if the volume of $\Omega_t$ does not depend on $t$. If $\{ \Omega_t \}_{t \in (-t_0, t_0)}$ is a deformation of $\Omega_0$, and $\lambda_{\Omega_t}$ and $u_t$ are respectively the first eigenvalue and the first eigenfunction (normalized to be positive and have $L^2 (\Omega_t)$ norm equal to $1$) of $-\Delta_{g}$ on $\Omega_t$ with $0$ Dirichlet boundary condition, both applications $t \longmapsto \lambda_{\Omega_t}$  and $t \longmapsto u_t$  inherit the regularity of the deformation of $\Omega_0$. These facts are standard and follow at once from the implicit function theorem together with the fact that the least eigenvalue of the Laplace-Beltrami operator with 0 Dirichlet boundary condition is simple. 

\medskip

A domain $\Omega_{0}$ is an \textit{extremal domain} (for the first eigenvalue of $-\Delta_{g}$ with 0 Dirichlet boundary condition) if for any volume preserving deformation $\{{\Omega}_t\}_{t \in (-t_0, t_0)}$ of ${\Omega}_{0}$, we have 
\[
\left.\frac{\textnormal{d} \lambda_{\Omega_t}}{\textnormal{d}t} \right|_{t =0} = 0 \, .
\] 




\medskip
 
Assume that $\{{\Omega}_t\}_t$ is a perturbation of a domain $\Omega_{0}$ generated by the vector field $\Xi$. The outward unit normal vector field to $\partial \Omega_{t}$ is denoted by $\nu_t$. We have the following result, whose proof can be found in \cite{ElS-Ilias} or in \cite{Pac-Sic}:
\begin{Proposition}\label{lambda} (Garabedian -- Schiffer, El Soufi -- Ilias). 
The derivative of the first eigenvalue with respect to the deformation of the domain is given by
\[
\left.\frac{\textnormal{d} \lambda_{\Omega_t}}{\textnormal{d}t} \right|_{t =0} = - \int_{\partial \Omega_{0}}  \left( g(\nabla  u_0 ,  \nu_0) \right)^{2}\  g(\Xi, \nu_0) \, \mbox{\rm dvol}_{g_\textnormal{in}}
\]
\end{Proposition}

This result allows to characterize extremal domains as the domains where there exists a positive solution to the overdetermined elliptic problem
\begin{equation}
\left\{
\begin{array}{rclll}
	\displaystyle \Delta_{g} u + \lambda \, u & = & 0 & \textnormal{in} & \Omega \\[3mm]
	\displaystyle u & = & 0 & \textnormal{on} & \partial \Omega \\[3mm]
	\displaystyle g (\nabla  u , \nu) & = & \textnormal{constant} & \textnormal{on} & \partial\Omega  \,
\end{array}
\right.
\end{equation}
for a positive constant $\lambda$, where $\nu$ is the outward unit normal vector about $\partial \Omega$. The proof of this fact follows directly from Proposition \ref{lambda}, but can be found also in \cite{Pac-Sic}.

\medskip

Given a point $p\in M$ we denote by $E_{1} ,\ldots, E_n$ an orthonormal basis of the tangent plane $T_p M$. Geodesic normal coordinates $x : =(x^1, \ldots, x^n) \in \mathbb R^n$ at $p$ are defined by 
\[
X (x) : =\mbox{Exp}_p^g \left( \sum_{j=1}^n x^j \, E_j \right) \in M
\]
where $\mbox{Exp}_p^g$ is the exponential map at $p$ for the metric $g$. 

\medskip

It will be convenient to identify $\mathbb R^n$ with $T_pM$ and $S^{n-1}$ with the unit sphere in $T_pM$. If  $x : = (x^1, \ldots, x^n) \in \mathbb R^n$, we set
\begin{equation}\label{nottheta}
\Theta (x) : = \sum_{i=1}^n x^i \, E_i \in T_pM \, .
\end{equation}
It corresponds to the vector of $T_p M$ whose geodesic normal coodinates are $x$. 
Given a continuous function $f : S^{n-1} \longmapsto (0, +\infty)$ whose $L^\infty$-norm is small (say less than the cut locus of $p$)  we define 
\[
B_{f}^g (p) : =  \left\{ \mbox{Exp}_p^g (  \Theta (x)) \qquad : \qquad  x \in \mathbb R^{n} \qquad 0 \leq |x|   < f \left(\frac{x}{|x|}\right) \right\} \, . 
\]
For notational convenience, given a continuous function $f : S^{n-1} \to (0, \infty)$, we set 
\[
B_{f} : =  \left\{  x \in \mathbb R^n  \qquad :  \qquad 0 \leq |x|   < f (x/|x|) \right\} \, . 
\]
When we do not indicate the metric as a superscript, we understand that we are using the Euclidean one. Similarly, we denote by $\textnormal{Vol}_g$ the volume in the metric $g$, by $\textnormal{dvol}_g$ the volume element in the metric $g$ to integrate over a domain, by $\textnormal{dvol}_{g_{\textnormal{in}}}$ the volume element in the induced metric $g_{\textnormal{in}}$ to integrate over the boundary of a domain. When we do not indicate anything we understand that we are considering the Euclidean volume, or the Euclidean measure, or the measure induced by the Euclidean one on boundaries.

\medskip

Our aim is to show that, for all  $\epsilon >0$ small enough, we can find a point $p\in M$ and a function $v : S^{n-1} \longrightarrow \mathbb R$  such that 
\[
{\rm Vol}_g \, B_{\epsilon(1+v)}^g(p) = {\rm Vol}\, B_\epsilon =  \epsilon^n \, {\rm Vol}\, B_1 =   \epsilon^n \,  \frac{\omega_n}{n}
\]
(where $\omega_n$ is the Euclidean volume of the unit sphere $S^{n-1}$) and the overdetermined problem
\begin{equation}\label{a1}
\left\{
\begin{array}{rcccl}
	\Delta_{g} \, \phi + \lambda \, \phi & = & 0 & \textnormal{in} & B_{\epsilon(1+v)}^g (p) \\[3mm]
	\phi & = & 0 & \textnormal{on} & \partial B^g_{\epsilon(1+v)}(p) \\[3mm]
	\displaystyle  g( \nabla  \phi ,  \nu) & = & {\rm constant} & \textnormal{on} & \partial B_{\epsilon(1+v)}^g (p) 
\end{array}
\right.
\end{equation}
has a non trivial positive solution for some positive constant $\lambda$, where $\nu$ is the unit normal vector field about $\partial B_{\epsilon(1+v)}^g(p)$. 
\medskip

Clearly, this problem does not make sense when $\epsilon = 0$. In order to bypass this problem, we observe that, considering the dilated metric $\bar g : =  \epsilon^{-2} \, g$, the above problem is equivalent to finding a point $p\in M$ and a function $v : S^{n-1} \longrightarrow \mathbb R$  such that 
\[
{\rm Vol}_{\bar g} \, B_{1+v}^{\bar g}(p) =  {\rm Vol} \, B_{1}
\]
and for which the overdetermined problem
\begin{equation}\label{a1}
\left\{
\begin{array}{rcccl}
	\Delta_{\bar g} \, \bar \phi + \bar \lambda \, \bar \phi & = & 0 & \textnormal{in} & B_{1+v}^{\bar g} (p) \\[3mm]
	\bar \phi & = & 0 & \textnormal{on} & \partial B^{\bar g}_{1+v}(p) \\[3mm]
	\displaystyle  \bar g( \nabla^{\bar g}  \bar \phi , \bar \nu ) & = & {\rm constant} & \textnormal{on} & \partial B_{1+v}^{\bar g} (p) 
\end{array}
\right.
\end{equation}
has a non trivial positive solution for some positive constant $\bar \lambda$, where $\bar \nu$ is the unit normal vector field about $\partial B_{1+v}^{\bar g}(p)$.  Taking in account that the functions $\phi$ and $\bar \phi$ have $L^2$-norm equal to 1, we have that the relation between the solutions of the two problems is simply given by 
\[
\phi = \epsilon^{- n/2} \, \bar \phi
\]
and 
\[
\lambda = \epsilon^{-2}Ê\, \bar \lambda \, .
\]

\section{Some expansions in normal geodesic coordinates}

We precise that through this paper we consider the following definition of the Riemann curvature tensor:
\[
Riem(X,Y)Z = \nabla_X\nabla_YZ-\nabla_Y\nabla_XZ-\nabla_{[X,Y]}Z
\]
where $\nabla$ denotes the Levi-Civita connection on the manifold $M$.

\medskip

Geodesic normal coordinates are very useful because there exists a well known formula for the expansion of the coefficients of a metric near the center of such coordinates, see \cite{willmore}, \cite{lee} or \cite{schoen}. At the point of coordinate $x$, the following expansion holds\footnote{We choose the convention of \cite{willmore}, some sign in the development  are different from those in \cite{Pac-Xu} or \cite{Pac-Sic} because
of a different  choice of the definition of $R_{ijkl}$}:
\begin{equation}\label{expmetric}
\begin{array}{lll}
g_{ij} & = & \displaystyle \delta_{ij} - \frac{1}{3} \, \sum_{k,\ell} R_{ikj\ell} \, x^{k} \, x^{\ell} - \frac{1}{6} \,\sum_{k, \ell, m}  R_{ikjl,m} \, x^{k} \, x^{\ell} \, x^{m}\\[4mm]
& &  \displaystyle- \frac{1}{20} \,\sum_{k, \ell, m, \sigma}  R_{ikjl,m\sigma} \, x^{k} \, x^{\ell} \, x^{m}\, x^{\sigma} + \frac{2}{45} \, \sum_{k,\ell, m, \sigma} R_{ikj\ell} \,  R_{imj\sigma} \, x^{k} \, x^{\ell} \, x^{m}\, x^{\sigma}  + {\mathcal O}(|x|^{5})
\end{array}
\end{equation}
where
\begin{eqnarray*}
R_{ikj\ell} & = & g\big( Riem_p(E_{k}, E_{i}) \, E_{j} ,E_{\ell}\big)\\[3mm]
R_{ikj\ell , m} & = & g\big((\nabla_{E_{m}} Riem)_p (E_{k}, E_{i}) \, E_{j} , E_{\ell} \big) \,
\\[3mm]
R_{ikj\ell , m\sigma} & = & g\big((\nabla_{E_{\sigma}} \nabla_{E_{m}} Riem)_p (E_{k}, E_{i}) \, E_{j} , E_{\ell} \big) \, ,
\end{eqnarray*}
and the subscript $p$ means that we evaluate the quantity at $p$. In (\ref{expmetric}) the Einstein notation is used (i.e.,  we do a summation on every index appearing up and down). Such notation will be always used through this paper. 

\medskip 

This expansion allows to obtain other expansions, as those of the volume of a geodesic ball, or the first eigenvalue and the first eigenfunction on a geodesic ball. In order to recall such expansions, let us introduce some notations. Let us denote by $\lambda_{1}$ the first eigenvalue of the Laplacian in the unit ball $B_{1}$ with $0$ Dirichlet boundary condition.  We denote by $\phi_{1}$ the associated eigenfunction 
\begin{equation}
\left\{
\begin{array}{rcccl}
	\Delta \phi_{1} + \lambda_{1}\, \phi_{1} & = & 0 & \textnormal{in} & B_{1} \\[3mm]
	\phi_{1} & = & 0 & \textnormal{on} & \partial B_{1} 
\end{array}
\right.
\label{eq:11-11}
\end{equation}
normalized to be positive and have $L^2 (B_{1})$ norm equal to $1$. It is clear that $\phi_1$ is a radial function $\phi_1(x) = \phi_1(|x|)$. We denote $r = |x|$.

\medskip

We recall now some expansions we will need later, whose proofs can be deduced from (\ref{expmetric}). We remind to \cite{Pac-Xu} and \cite{Karp-Pinsky} for the proofs. For the volume of a geodesic ball of radius $\epsilon$ we have:
\begin{equation}\label{expvol}
\epsilon^{-n}\, \textnormal{Vol}_g\, B^g_{\epsilon}(p)=\displaystyle \frac{\omega_n}{n}\, + W_0\, \epsilon^2
+W \, \epsilon^4+
 \mathcal{O}(\epsilon^5),
\end{equation}
where 
\begin{equation}\label{W}
\begin{array}{lll}
W_0&=&\displaystyle- \frac{\omega_n}{6n\,(n+2)}\, R_p\\[4mm]
W&=&\displaystyle \frac{\omega_n}{360\, n\,(n+2)\,(n+4)}\,\left(-3\,\|Riem_p\|^2+8\,\|Ric_p\|^2+5\,R_p^2-18\,(\Delta_gR)_p\right)
\end{array}
\end{equation}
For the first eigenvalue of the Laplace-Beltrami operator with 0 Dirichlet boundary condition on a geodesic ball of radius $\epsilon$ we have: 
\begin{equation}\label{expeig}
\displaystyle \epsilon^2\, \lambda_{B_\epsilon^g(p)}=  \displaystyle \lambda_1+ \Lambda_0\, \epsilon^2\, 
+ \Lambda \, \epsilon^4
+ \mathcal{O}(\epsilon^5)
\end{equation}
where 
\begin{equation}\label{constL}
\begin{array}{lll}
\Lambda_0&=&  - \displaystyle \frac{R_p}{6}\\[4mm]
\Lambda&=& \displaystyle -\frac{c^2}{n(n+2)}\, \left(3\, \|Riem_p\|^2+\frac{35}{18}\, \|Ric_p\|^2+\frac{5n-3}{18n}\, R_p^2+\frac15(\Delta_g R)_p\right)
\end{array}
\end{equation}
and the constant $c^2$ is given by
\[
c^2 = -\int_0^1\phi_1\, \partial_r\phi_1\,  r^{n+2}\, \textnormal{d}r=\frac{n+2}2\int_0^1\phi_1^2\, r^{n+1}\, \textnormal{d}r
\]

For the associate eigenfunction $\phi$ in the geodesic ball $B^g_\epsilon(p)$ normalized to be positive and with $L^2$-norm equal to 1, we have
\begin{equation}\label{phigeod}
\epsilon^{n/2}\, \phi (q) =  \phi_1 (y) + \left[\left(R_{ij}\,y^i\,y^j-\frac R n |y|^2\right)\frac{\phi_1}{12}+R\;G_2(|y|)\right]\, \epsilon^2+\mathcal{O}(\epsilon^3)
\end{equation}
where $q$ is the point of $M$ whose geodesic coordinates are $\epsilon\, y$ for $y \in B_1$, and $G_2$ is defined implicitly as a solution of an ODE in \cite{Karp-Pinsky}.
Although we do not need its expression, for completeness we recall it: if we solve such ODE we found 
\begin{equation}\label{g2}
G_2(r)=\frac1{12\, n}\, r^2\, \phi_1(r)-c^2\, \frac{\omega_n}{6n\, (n+2)}\,  \phi_1(r)\, .
\end{equation}



 \medskip

\section{Known results}

Our aim is to perturbe the boundary of a small ball $B^{\bar g}_1(p)$ with a function $v$ in order to obtained an extremal domain $B^{\bar g}_{1+v}(p)$. The natural space for the function $v$ is $C^{2,\alpha}(S^{n-1})$ but not all functions in this space are admissible because $v$ must satisfy also the condition
\[
\textnormal{Vol}_{\bar g}\, B^{\bar g}_{1+v}(p) = \textnormal{Vol}\, B_1
\]
In order to have a space of admissible functions not depending on the point $p$, we use a result proved in \cite{Pac-Sic}, that allows to use as space of admissible function the space
\[
C^{2,\alpha}_m(S^{n-1}) = \left\{ \bar v \in C^{2,\alpha}(S^{n-1}) \qquad : \qquad \int_{S^{n-1}} \bar v  =0 \right\}
\]

The result is the following:
\begin{Proposition} (Pacard -- Sicbaldi \cite{Pac-Sic})
\label{pr:1.2}
Let $p \in M$. For all $\epsilon$ small enough and all function $\bar v \in C^{2,\alpha}_m(S^{n-1})$ whose $C^{2, \alpha}$-norm is small enough there exist a unique positive function $\bar \phi = \bar \phi (p, \epsilon, \bar v) \in C^{2, \alpha} (B_{1+v}^{\bar g} (p))$, a constant $\bar \lambda = \bar \lambda (p,\epsilon, \bar v) \in \mathbb R$ and a constant $v_0 = v_0 (p, \epsilon, \bar v) \in \mathbb R$ such that
\[
{\rm Vol}_{\bar g}\,  B^{\bar g}_{1+v} (p) = {\rm Vol}\, B_{1}
\]
where $v : =  v_0 + \bar v$ and $\bar \phi$ is a  solution to the problem
\begin{equation}\label{formula}
\left\{
\begin{array}{rcccl}
	\Delta_{\bar g} \, \bar \phi + \bar \lambda \, \bar \phi & = & 0 & \textnormal{in} & B_{1+v}^{\bar g}(p) \\[3mm]
	\bar \phi & = & 0 & \textnormal{on} & \partial B_{1+v}^{\bar g} (p)
\end{array}
\right.
\end{equation}
normalized by 
\[
\int_{B_{1+v}^{\bar g} (p)} \bar \phi^2 \, \textnormal{dvol}_{\bar g} =1 .
\label{noral}
\]
In addition $\bar \phi$, $\bar \lambda$ and $v_0$ depend smoothly on the function $\bar v$ and the parameter $\epsilon$ and $\bar \phi = \phi_1$, $\bar \lambda =\lambda_1$ and $v_0 = 0$ when $\epsilon =0$ and $\bar v\equiv 0$. Moreover $v_0(p, \epsilon, 0) = \mathcal{O}(\epsilon^2)$.
\end{Proposition}

Instead of working on a domain depending on the function $v = v_0 + \bar v$, it will be more convenient to work on a fixed domain $B_1$ endowed with a metric depending on both $\epsilon$ and the function $v$. This can be achieved by considering the parametrization of $B^{\bar g}_{1+v}(p)$ given by 
\[
Y ( y) : =\mbox{Exp}_p^{\bar g} \left( \left(1 + v_0 + \chi(y)\, \bar v \left(\frac{y}{|y|} \right)  \right) \, \sum_i y^i \, E_i \right)
\]
where $\chi$ is a cutoff function identically equal to $0$ when $|y| \leq 1/2$ and identically equal to $1$ when $|y|Ê\geq 3/4$. Hence the coordinates we consider from now on are $y \in B_1$ with the metric $\hat g : = Y^* \bar g$.

\medskip
 
Up to some multiplicative constant, the problem we want to solve can now be rewritten in the form 
\begin{equation}
\label{formula-new}
\left\{
\begin{array}{rcccl}
	\Delta_{\hat{g}} \, \hat \phi + \hat \lambda \, \hat \phi & = & 0 & \textnormal{in} & B_1 \\[3mm]
	\hat \phi & = & 0 & \textnormal{on} & \partial B_1 
\end{array}
\right.
\end{equation}
with 
\begin{equation}
\label{formula-new-1}
\int_{B_1} \hat \phi^2 \, \mbox{dvol}_{\hat g} =1
\end{equation}
and 
\begin{equation}
\label{formula-new-2}
{\rm Vol}_{\hat g}Ê(B_1) ={\rm Vol}\, B_1
\end{equation}
When $\epsilon =0$ and $\bar v \equiv 0$, a solution of (\ref{formula-new}) is given by $\hat \phi = \phi_1$, $\hat \lambda = \lambda_1$ and $v_0 =0$.   In the general case, the relation between the function $\bar \phi$ and the function $\hat \phi$ is simply given by
\[
Y^* \bar \phi  =  \hat \phi  \qquad \mbox{and} \qquad  \bar \lambda =  \hat \lambda \, .
\]

\medskip

After canonical identification of $\partial B_{1+v}^{\bar g} (p)$ with $S^{n-1}$, we define the operator
\[
F (p, \epsilon, \bar v) =  \displaystyle \bar g (\nabla \bar \phi ,  \bar \nu )  \, |_{\partial B_{1+v}^{\bar g}}   - \frac{1}{{\rm Vol}_{{\bar g}_{\textnormal{in}}} (\partial B_{1+v}^{\bar g})} \, \int_{\partial B_{1+v}^{\bar g}}Ê \, \bar g (\nabla \bar \phi ,  \bar \nu ) \, \mbox{dvol}_{{\bar g}_{\textnormal{in}}}  \, ,
\]
where $\bar \nu$ denotes the unit normal vector field to $\partial B_{1+v}^{\bar g}$ and {$(\bar \phi, v_0)$} is the solution of (\ref{formula}) provided by the Proposition \ref{pr:1.2}. Recall that $v = v_0 + \bar v$. Schauder's estimates imply that {$F$} is well defined from a neighbourhood of $M \times \{0\} \times \{0\}$ in  $M \times [0,\infty) \times C^{2,\alpha}_m (S^{n-1})$ into $C_m^{1,\alpha}(S^{n-1})$ (the space $C_m^{1,\alpha}(S^{n-1})$ is naturally the space of functions in $C^{1,\alpha}(S^{n-1})$ whose mean is 0). Our aim is to find $(p,\epsilon,\bar v)$ such that $F(p,\epsilon,\bar v)=0$. Observe that, with this condition, $\bar \phi$ will be the solution to problem (\ref{a1}).

\medskip

We also have the alternative expression for $F$ using the coordinates of $B_1$ and the metric $\hat g$:
\[
F (p, \epsilon, \bar v) =  \displaystyle \left.\hat g (\nabla \hat \phi ,  \hat \nu )  \, \right|_{\partial B_{1}}   -\frac{1}{ {\rm Vol}_{{\hat g}_{\textnormal{in}}} (\partial B_{1}) } \,  \int_{\partial B_{1}}Ê \, \hat g (\nabla \hat \phi ,  \hat \nu ) \, \mbox{dvol}_{{\hat g}_{\textnormal{in}}} \, 
\]
where this time $\hat \nu$ is the  the unit normal vector field to $\partial B_{1}$ using the metric $\hat g$. 

\medskip 

For all $ \bar v \in \mathcal C^{2, \alpha}_m (S^{n-1})$ let $\psi$ be the (unique) solution of 
\begin{equation}\label{c00}
\left\{
\begin{array}{rcll}
	\displaystyle \Delta \psi + \lambda_{1}Ê\,Ê \psi & = & 0  & \textnormal{in} \qquad B_1 \\[3mm]
	\psi & = &  - \displaystyle  {c_1} \, \bar v  & \textnormal{on}\qquad \partial B_1 \, 
\end{array}
\right.
\end{equation}
which is $L^2(B_1)$-orthogonal to $\phi_1$, where $c_1 : = \left. \partial_r \phi_1\right|_{r=1}$. Define
\begin{equation}
H(\bar v) : = \left.\left( {\partial_r \psi }  + {c_2} \, \bar v \right) \, \right|_{\partial B_1}
\label{Hache}
\end{equation}
where $c_2 = \left.\partial^2_r \phi_{1}\right|_{r=1}$. We recall that the eigenvalues of the operator $-\Delta_{S^{n-1}}$ are given by $\mu_{j} = j \, (n-2+j)$
for $j \in \mathbb N$, and we denote by $V_j$ the eigenspace associated to $\mu_j$. 

\medskip

The following result shows that $H$ is the linearization of $F$ with respect to $\bar v$ at $\epsilon =0$ and $\bar v=0$:

\begin{Proposition} (Pacard -- Sicbaldi, \cite{Pac-Sic})
\label{H}
The operator obtained by linearizing $F$ with respect to $\bar v$ at $\epsilon =0$ and $\bar v=0$ is
\[
H : C^{2, \alpha}_m (S^{n-1}) \longrightarrow C^{1, \alpha}_m (S^{n-1})
\]
It is a self adjoint, first order elliptic operator. The kernel of $H$ is given by $V_1$. Moreover there exists $c >0$ such that 
\[
\| w \|_{\mathcal C^{2, \alpha}(S^{n-1})}  \leq c \, \| H(w)  \|_{\mathcal C^{1, \alpha}(S^{n-1})} \, ,
\]
provided $w$ is $L^2(S^{n-1})$-orthogonal to $V_0 \oplus V_1$.
\end{Proposition}

\medskip

Using the previous proposition, the implicit function theorem gives directy the following:

\begin{Proposition}\label{solmodnoyau} (Pacard -- Sicbaldi, \cite{Pac-Sic})
There exists $\epsilon_0 >0$ such that, for all $\epsilon \in [0, \epsilon_0]$ and for all $p \in M$, there exists a unique function $\bar v = \bar v(p, \epsilon) \in C^{2,\alpha}_m(S^{n-1})$, orthogonal to $V_0 \oplus V_1$, and a vector $a  = a (p, \epsilon) \in \mathbb R^n$ such that 
\begin{equation}\label{Fa}
F (p, \epsilon, \bar v )  + \langle a , \cdot \rangle = 0 
\end{equation}
The function $\bar v$  and the vector $a$ depend smoothly on $p$ and $\epsilon$ and we have 
\[
|Êa | + \| \bar v \|_{\mathcal C^{2, \alpha}(S^{n-1})} \leq c \, \epsilon^2
\]
\end{Proposition}

\medskip

In other word, for every point $p \in M$ it is possible to perturbe the small ball $B^{\bar g}_1(p)$ in a domain $B^{\bar g}_{1+v}(p)$, whose volume did not change, but with the (strong) property that $F(p,\epsilon, \bar v)$ (i.e. the Neumann data of its first eigenfunction minus its mean) is the restriction of a linear function $\langle a , \cdot \rangle$ on $S^{n-1}$. It is important to underline that this result does not depend on the geometry of the manifold, because it is true for every point $p$. 

\medskip 

Now, we have to find the good point $p$ for which such linear function $\langle a , \cdot \rangle$ is the 0 function. And in this research we will see the geometry of the manifold.

\section{Construction of small extremal domains}

For $p \in M$, let us define the function
$$
\Psi_\epsilon(p):=\hat\lambda = \hat \lambda(p, \epsilon, \bar v(p, \epsilon))
$$
where $\hat \lambda$ is given by (\ref{formula-new}) taking $\bar v = \bar v(p,\epsilon)$ given by Proposition \ref{solmodnoyau}.

\begin{Proposition}\label{critique}
For $\epsilon$ small enough, the domain $B_{1+v(p,\epsilon)}^{\bar g} (p)$ is extremal if and only if $p$ is a critical point of $\Psi_{\epsilon}$, where $v(p, \epsilon) = v_0(p, \epsilon, \bar v(p, \epsilon)) + \bar v(p, \epsilon)$.
\end{Proposition}

\begin{proof}
Recall that by definition 
$$
F(p,\epsilon,\bar v(p,\epsilon))=\hat{g} (\hat{\nu},\hat {\nabla}\hat{\phi})-b
$$
where
$$
b= b(p, \epsilon) := \frac 1{ {\rm Vol}_{{\hat g}_{\textnormal{in}}} (\partial B_{1}) }\int_{\partial {B}_1}\hat{g} (\hat{\nu},\hat {\nabla}\hat{\phi})\, \mbox{dvol}_{{\hat g}_{\textnormal{in}}}
$$
and
$$
\int_{\partial {B}_1}F\, \mbox{dvol}_{{\hat g}_{\textnormal{in}}}=0 \, .
$$
Moreover we know that
$$
F(p,\epsilon,\bar v(p,\epsilon))+\langle a(p,\epsilon),\cdot \rangle=0.
$$
In particular the domain $B_{1+v(p,\epsilon)}^{\bar g} (p)$ is extremal if and only if $a(p,\epsilon)=0$.

\medskip

Let us now compute the differential of $\Psi_{\epsilon}$.
Let $\Xi\in T_pM$ and $$q:=\mbox{Exp}_p(t\Xi).$$
For $t$ small enough, the boundary of $B_{1+v(q,\epsilon)}^{\bar g} (q)$ can be written as a normal graph
over the boundary of $B_{1+v(p,\epsilon)}^{\bar g} (p)$ for some function $f$, depending on $p,\epsilon,t$ and $\Xi$, and smooth on $t$.
This defines a vector field on $\partial B_{1+v(p,\epsilon)}^{\bar g} (p)$ by
$$
Z:=\left.\frac{\partial f}{\partial t}\right|_{t=0}\, \bar{\nu}
$$
where $\bar{\nu}$ is the normal of $\partial B_{1+v(p,\epsilon)}^{\bar g} (p)$.
Let $X$ be the parallel transported of $\Xi$ from geodesic issued from $p$. As the metric $\bar g$ is close to the Euclidean one for $\epsilon$ small, there exists a constant $c$ such that for all $\epsilon$ small enough and any $\Xi$ the estimation 
$$
\|Z-X\|_{\bar{g}}\leq c\|\Xi\|_{\bar{g}}.
$$
holds.
The  variation of the first eigenvalue, see Proposition \ref{lambda}, gives
$$
D_p\Psi_\epsilon(\Xi)= \left.\frac{\mbox{d}}{\mbox{d}t}\right|_{t=0}\, \Psi_\epsilon(q)=  -\int_{\partial B_1}[\hat{g}(\hat\nabla\hat\phi,\hat\nu)]^2\, \hat g(\hat Z,\hat\nu)\, \mbox{dvol}_{{\hat g}_{\textnormal{in}}}.
$$
We thus obtain
\begin{equation}\label{dpsi}
D_p\Psi_\epsilon(\Xi)=-\int_{\partial B_1}[-\langle a(\epsilon,p),\cdot\rangle+b]^2\, \hat g(\hat Z,\hat\nu)\, \mbox{dvol}_{{\hat g}_{\textnormal{in}}}
\end{equation}
Recall that the variation we made is volume preserving, i.e.
$$
\int_{\partial B_1}\hat g(\hat Z,\hat\nu)\, \mbox{dvol}_{{\hat g}_{\textnormal{in}}}=0 \, .
$$
Then it is easy to see that if $a=0$ then $D_p\Psi_\epsilon=0$. This proves one implication.

\medskip

For the reverse implication, assume now that $D_p\Psi_\epsilon=0$. From (\ref{dpsi}) we have
\begin{equation}
\label{equaint}
2b\, \int_{\partial B_1}\langle a(\epsilon,p),\cdot\rangle\, \hat g(\hat Z,\hat\nu)\, \mbox{dvol}_{{\hat g}_{\textnormal{in}}}=\int_{\partial B_1}\langle a(\epsilon,p),\cdot\rangle^2\,\hat g(\hat Z,\hat\nu)\,\mbox{dvol}_{{\hat g}_{\textnormal{in}}}
\end{equation}
for all $\Xi$. 
It is easy to see that for all $\epsilon$ small enough there exists a constant $c$ such that
$$
|\hat g(\hat Z,\hat\nu)-\langle \Xi,\cdot\rangle|\leq c\, \epsilon\, \|\Xi\|_g
$$
(in fact the left hand side vanishes when $\epsilon=0$, the metric $\hat g$ and the Euclidean one differ by terms of order $\epsilon^2$ and the normal vectors differ by terms of order $\epsilon$).
Now we choose $\Xi=b\, a = b(p, \epsilon) \, a(p, \epsilon)$ and we get
$$
\hat g(\hat Z,\hat\nu)=b\, \langle a,\cdot \rangle +\epsilon\, A
$$
where $|A|\leq c\,\|ba\|_g$.
Using this equality in equation~(\ref{equaint}), we deduce that for all $\epsilon$ small enought there exists a constant $C$ independent on $\epsilon$ and $a$ such that
$$
2b^2\, \int_{\partial B_1}\langle a,\cdot \rangle^2\, \mbox{dvol}_{{\hat g}_{\textnormal{in}}}\leq C\, |b|\, (\epsilon\, \|a\|^3+\|a\|^3+\epsilon\|a\|^2) \, .
$$
Now the left hand side is bounded by below by $b^2\, \|a\|^2$, so finally we obtain
$$
b^2\, \|a\|^2\leq C\, |b|\, (\epsilon\,  \|a\|+\|a\|+\epsilon)\, \|a\|^2 \, .
$$
Observe that we cannot have  $b=0$ because it would imply that 
$$
\hat\lambda\int_{ B_1}\hat\phi\;\mbox{dvol}_{\hat g}=\int_{B_1}\Delta_{\hat g}\hat\phi\;\mbox{dvol}_{\hat g}=b=0
$$
but $\hat\phi>0$ on $B_1$ so  $\hat\lambda=0$ which is not possible because it is the first eigenvalue of the Laplace-Beltrami operator with 0 Dirichlet boundary condition. As $\|a\| = \mathcal{O}(\epsilon^2)$, then for $\epsilon$ small (recall $b\neq0$) we obtain that $a=0$ and this concludes the proof of the proposition. 
\end{proof}

We now define 
$$
\Phi(p,\epsilon)=-\frac{6\, n\, (n+2)}{n\, (n+2)+2\lambda_1} \, \frac{\Psi_\epsilon(p)-\lambda_1}{\epsilon^2} \, .
$$
Propositions \ref{solmodnoyau} and \ref{critique} completes the proof of the first part of Theorem~\ref{maintheorem}. In the following sections, we will prove the second and the third parts of Theorem~\ref{maintheorem}, and for this we have to find an expansion in power of $\epsilon$ for $\Psi_\epsilon(p)$. Such expansion will involve the geometry of the manifold. 

\section{Expansion of the first eigenvalue on perturbations of small geodesic balls}


In this section we want to find an expansion of the first eigenvalue $\hat \lambda = \hat \lambda(p, \epsilon, \bar v)$ in power of $\epsilon$ and $\bar v$, where $p$ is fixed in $M$. In a second time, we will use the function $\bar v = \bar v(p, \epsilon)$ given by Proposition \ref{solmodnoyau} in order to find an expansion of $\hat \lambda(p, \epsilon, \bar v(p, \epsilon))$ in power of $\epsilon$.
Keeping in mind that we will have $\bar v=\mathcal{O}(\epsilon^2)$ we write formally
\[
\begin{array}{lll}
\displaystyle \hat \lambda (p,\epsilon,\bar v) & = & \hat \lambda (p,0,0)+\partial_\epsilon\hat \lambda (p,0,0)\,\epsilon\\[4mm]
&&\displaystyle+ \partial_{\bar v}\hat \lambda (p,0,0)\, \bar v+\frac12\, \partial_\epsilon^2\hat \lambda (p,0,0)\, \epsilon^2\\[4mm]
&   &\displaystyle+\partial_\epsilon\partial_{\bar v}\hat \lambda (p,0,0)\, \epsilon\, \bar v
+\frac16\, \partial_\epsilon^3\hat \lambda (p,0,0)\, \epsilon^3\\[4mm]

&&\displaystyle+\frac12\, \partial_{\bar v}^2\hat \lambda (p,0,0)\, \bar v^2+\frac12\, \partial_\epsilon^2\partial_{\bar v}\hat \lambda (p,0,0)\, \epsilon^2\, \bar v
+\frac1{24}\, \partial_\epsilon^4\hat \lambda (p,0,0)\, \epsilon^4\\[4mm]
&&\displaystyle+\mathcal{O}(\epsilon^5)
\end{array}
\]

We thus study all of theses terms.

\begin{Lemma}\label{exp1} 
We have 
\[
\begin{array}{lll}
\partial_\epsilon\hat \lambda (p,0,0) &=&0 \\[4mm]
\displaystyle \frac12\, \partial_\epsilon^2 \hat \lambda (p,0,0) &=& \displaystyle - \frac{R_p}{6}\, \left(1 + 2\,\frac{ \lambda_1}{ n\, (n+2)}\right) = : \hat \Lambda_0\\[4mm]
\partial_\epsilon^3 \hat \lambda (p,0,0)& = &0\\[4mm]
\displaystyle \frac1{24}\, \partial_\epsilon^4 \hat \lambda (p,0,0) & = &  \displaystyle \Lambda + \lambda_1\, \left( \frac{2\, W}{\omega_n} - \frac{R_p^2}{36\, n^2\, (n+2)} \right) =: \hat \Lambda
\end{array}
\]
where 
the constants $\Lambda$ and $W$ are given in (\ref{W}) and (\ref{constL}).
\end{Lemma}

\begin{proof}
It suffices to find the expansion of $ \hat \lambda (p,\epsilon,0)$ in power of $\epsilon$. First we have to expand $v_0(p, \epsilon,0)$ and this can be done by using expansion (\ref{expvol}), keeping in mind the definition of the metric $\hat g$ and the fact that when $\bar v = 0$ the constant $v_0$ is given by the relation
\[
\textnormal{Vol}_{\hat g}\, B_1 = \textnormal{Vol}\, B_1 = \frac{\omega_n}{n}
\]
We find
\[
v_0 = A_0\, \epsilon^2 + A\, \epsilon^4 
+\mathcal{O}(\epsilon^5)\\
\]
where
$$
\begin{array}{lll}
A_0 &=& - \displaystyle \frac{W_0}{\omega_n}\\[4mm]
A & = &- \displaystyle \frac{1}{\omega_n} \, \left( (n+2)\, A_0\, W_0 + \frac{n-1}{2}\, A_0^2\, \omega_n + W\right)
\end{array}
$$
Now we use expansion (\ref{expeig}) replacing $\epsilon$ by $\epsilon(1+v_0)$. 
We obtain 
\[
\hat \lambda = \lambda_1 + \hat \Lambda_0\, \epsilon^2 + \hat \Lambda\, \epsilon^4 
+\mathcal{O}(\epsilon^5)\\
\]
where 
$$
\begin{array}{lll}
\hat \Lambda_0 &=& \displaystyle \Lambda_0 - 2\, \lambda_1 \, A_0 = -\frac{R_p}{6}\, \left(1 + 2\,\frac{ \lambda_1}{ n\, (n+2)}\right)\\[4mm]
\hat \Lambda & = & \displaystyle \Lambda - \lambda_1\, (2A - 3\, A_0^2)  = \Lambda + \lambda_1\, \left( \frac{2\, W}{\omega_n} - \frac{R_p^2}{36\, n^2\, (n+2)} \right)
\end{array}
$$
This concludes the proof of the result. 
\end{proof}

\begin{Lemma}\label{exp2} 
We have 
\[
\begin{array}{lll}
\partial_{\bar v}\hat \lambda (p,0,0) & = & 0\\[4mm]
\partial_{\bar v}^2\hat \lambda (p,0,0)(\bar v,\bar v) & = & \displaystyle -2\,c_1\,\int_{S^{n-1}}\bar{v}\, H(\bar v)\, \\
\end{array}
\]
where $H$ is the operator of Proposition \ref{H}, whose expression is given by (\ref{Hache}), and $c_1 : = \left. \partial_r \phi_1\right|_{r=1}$ is the constant defined in (\ref{c00}).
\end{Lemma}

\begin{proof}
Let $\Omega_0= B_1$ be the unit ball of $\mathbb{R}^n$, and let $\Omega_t= B_{(1+v_0+t\bar v)}$,
where we recall that $\int_{S^{n-1}}\bar v=0$ and $v_0 = v_0(t)$ is chosen in order that $\textnormal{Vol}\,\Omega_t=\textnormal{Vol}\, \Omega_0 = \frac{\omega_n}{n}$.
We have
\[
\partial_{\bar v}\hat \lambda (p,0,0)(\bar v) = \left.\frac{\textnormal{d}}{\textnormal{d} t}\right|_{t=0} \lambda_{\Omega_t}
\]
and
\[
\partial_{\bar v}^2\hat \lambda (p,0,0)(\bar v,\bar v) = \left.\frac{\textnormal{d}^2}{\textnormal{d} t^2}\right|_{t=0}\lambda_{\Omega_t}
\]
where $\lambda_{\Omega_t}$ is the first Dirichlet eigenvalue of $\Omega_t$.
The expansion of $\textnormal{Vol}\, \Omega_t$ directly prove that $v_0=\mathcal{O}(t^2)$. In fact, in polar coordinates, we have
\begin{eqnarray*}
\textnormal{Vol}\, \Omega_t & = & \int_{S^{n-1}} \int_0^{1+v_0 + t\, \bar v} r^{n-1}\, \textnormal{d}r\, \textnormal{d}\theta\\
& = & \frac{1}{n} \int_{S^{n-1}} (1+v_0 + t\, \bar v)^n\, \textnormal{d}\theta\\
& = & \frac{1}{n} \int_{S^{n-1}} \left[(1+v_0)^n + n\, (1+v_0)^{n-1}\, t\, \bar v + \mathcal{O}(t^2)\right]\, \textnormal{d}\theta\\
& = & \frac{\omega_n\, (1+v_0)^n}{n} + \mathcal{O}(t^2)
\end{eqnarray*}
Differentiating this expression with respect to $t$, and keeping in mind that $v_0(0) = 0$, we obtain that $v_0(t) = \mathcal{O}(t^2)$.
For $y\in\Omega_0$ and $t$ small, let 
\[
h(t,y)=\left(1+v_0+t\, \chi(y)\, \bar v\left(\frac{y}{|y|}\right)\right)y
\]
where $\chi$ is a cutoff function identically equal to $0$ when $|y| \leq 1/2$ and identically equal to $1$ when $|y|Ê\geq 3/4$, so that 
$h(t,\Omega_0)=\Omega_t$. We will denote the $t$-derivative with a dot. Let $V(t,h(t,y))=\dot h(t,y)$ be the first variation of the domain $\Omega_t$.
Let $\nu$ be the unit normal to $\partial\Omega_t$ and let $\sigma= \langle V,  \nu\rangle $ the normal variation about $\partial \Omega_t$.
Let $\lambda$ be the first eigenvalue and $\phi$ the first eigenfunction of the Dirichlet Laplacian over $\Omega_t$ normalized in order to have $L^2$ norm equal to 1.
From Proposition \ref{lambda} we have 
$$
\dot \lambda=-\int_{\partial \Omega_t}\,(\partial_\nu \phi)^2\, \sigma\, 
$$
where $\partial_\nu \phi = \langle \nabla \phi, \nu\rangle$. At $t=0$ and on the boundary, we have $\phi=\phi_1$, $\partial_\nu\phi=\partial_r\phi_1=c_1$, $\sigma=\bar v$. Then $\dot \lambda (0) = 0$. This proves the first part of the Lemma.

\medskip

We can use now equality (\ref{henry2}) of Proposition \ref{prop_henry} of the Appendix (with $f = (\partial_\nu \phi )^2\, \sigma$) in order to derivate this formula with respect to $t$. We obtain
$$
\ddot \lambda=-\int_{\partial \Omega_t} \left[(\partial_\nu \phi)^2\, (\dot\sigma+\sigma\,\partial_\nu \phi\, + \tilde H\,\sigma^2) + 2\sigma\,(\partial_\nu \phi\,\partial_\nu \dot \phi+\sigma\,\partial_\nu \phi\, \partial^2_\nu \phi)\right]\, 
$$
where $\tilde H$ is the mean curvature of $\partial \Omega_t$. Now the second variation of the volume of $\Omega_t$ is
$$
\ddot {\textnormal{Vol}\, \Omega_t}=\int_{\partial \Omega_t}(\dot\sigma+\sigma\,\partial_\nu\sigma+\tilde H\sigma^2)=0.
$$
Such equation can be obtained differentiating equality (\ref{henry1}) of Proposition \ref{prop_henry} with $f=1$, using equality (\ref{henry1}) of Proposition \ref{prop_henry} with $f=\sigma$. 
On the other hand, at $t=0$ and on the boundary, we have $\phi=\phi_1$, $\partial_\nu\phi=\partial_r\phi_1=c_1$,
$\partial^2_\nu\phi=\partial^2_r\phi_1=c_2=-(n-1)c_1$, $\sigma=\bar v$, $\dot\phi=\psi$, where $\psi$ solve (\ref{c00}).
We obtain
$$
\ddot{\lambda}(0)=-2\, c_1\int_{S^{n-1}}\bar v\, (\partial_r\psi+c_2\bar v)=-2\, c_1\int_{S^{n-1}}\bar v\, H(\bar v).
$$
The proof of the Lemma follows at once.
\end{proof}

\begin{Lemma}\label{exp3}
We have
\[
\begin{array}{lll}
\partial_\epsilon\partial_{\bar v}\hat \lambda (p,0,0) &=&0 \\[4mm]
\partial_\epsilon^2\partial_{\bar v}\hat \lambda (p,0,0)\, \bar v &=&\displaystyle - \frac{c_1^2}{3} \int_{S^{n-1}}\mathring{Ric}_p(\Theta,\Theta)\, \bar v\, \\
\end{array}
\]
where $\Theta$ is the vector of $T_p M$ whose geodesic coordinates are $y \in S^{n-1}$, according with (\ref{nottheta}), and $c_1 : = \left. \partial_r \phi_1\right|_{r=1}$ is the constant defined in (\ref{c00}).

\end{Lemma}

In order to prove this lemma, we start with a preliminary result. The formulas for the geometric quantities we will consider are potentially complicated, and to keep notations short, we agree on the following: any expression of the form $L_{p}(v)$ denotes a linear combination of the function $v$ together with its derivatives up to order 1, whose coefficients can depend on $\epsilon$ and there exists a positive constant $c$ independent on $\epsilon \in (0,1)$ and on $p$ such that
\[
\|L_p(v)\|_{C^{1, \alpha}(S^{n-1})} \leq c\,  \|v\|_{C^{2, \alpha}(S^{n-1})} \, ;
\]
similarly, given $a \in \mathbb{N}$, any expression of the form $Q_p^{(a)}(v)$ denotes a nonlinear operator in the function $v$ together with its derivatives up to order 1,  whose coefficients can depend on $\epsilon$ and there exists a positive constant $c$ independent on $\epsilon \in (0,1)$ and on $p$
such that
\[
\|Q_p^{(a)}(v_1)-Q_p^{(a)}(v_2)\|_{C^{1, \alpha}(S^{n-1})} \leq c\, \left( \|v_1\|_{C^{2, \alpha}(S^{n-1})}  +  \|v_2\|_{C^{2, \alpha}(S^{n-1})}\right)^{a-1} \,  \|v_2-v_1\|_{C^{2, \alpha}(S^{n-1})}
\]
provided $ \|v_i\|_{C^{2, \alpha}(S^{n-1})} \leq 1$, for $i=1,2$.

\begin{Lemma}\label{v0punto}
We have
$$
\partial_{\bar v}v_0(p,\epsilon,0)(\bar v)=\frac{\epsilon^2}{6\,\omega_n}\, \int_{S^{n-1}}\mathring{Ric}(\Theta,\Theta)\, \bar v
+\int_{S^{n-1}}[\mathcal{O}(\epsilon^5)+\epsilon^3\, L_p(\bar v)]\,
$$
where $\Theta$ is the vector of $T_p M$ whose geodesic coordinates are $y \in S^{n-1}$, according with (\ref{nottheta}).

\end{Lemma}

\begin{proof}
The expansion in $\epsilon$ and $v$ for the volume of the perturbed geodesic ball $B^g_{\epsilon(1+v)}(p)$ is given in the Appendix of \cite{Pac-Xu} (the corresponding notations with respect to \cite{Pac-Xu} are $B^g_{\epsilon\,(1+v)}(p) = B_{p,\epsilon}(-v)$ and $n = m+1$). We have: 
\begin{equation}\label{volpacxu}
\begin{array}{lll}
\epsilon^{-n}\, \textnormal{Vol}(B^g_{\epsilon(1+v)}(p))&=&\displaystyle \frac{\omega_n}{n} + W_0\,  \epsilon^2 + W\, \epsilon^4 \\[4mm]
&&\displaystyle - \int_{S^{n-1}} v\,   +  \frac{n-1}2\, \int_{S^{n-1}} v^2\,  -\frac16\,\epsilon^2\, \int_{S^{n-1}}Ric_p(\Theta,\Theta)\, v\,  \\[4mm]
&&\displaystyle +\int_{S^{n-1}}\left(\mathcal{O}(\epsilon^5)+\epsilon^3\, L_p(v)+\epsilon^2\, Q_p^{(2)}(v)+Q_p^{(3)}(v)\right)
\end{array}
\end{equation}
where $\Theta$ is the vector in $T_pM$ whose coordinates are $y \in S^{n-1}$, and $W_0$, $W$ are given by (\ref{W}). Putting $v = v_0 + \bar v$ in expansion (\ref{volpacxu}), where $\int_{S^{n-1}} \bar v  = 0$ and $v_0$ is chosen in order that the volume of $B^g_{\epsilon\,(1+v)}(p)$ is equal to the volume of $B_\epsilon$, we obtain 
$$
\begin{array}{lll}
\epsilon^{-n}\, \textnormal{Vol}(B^g_{\epsilon(1+v)}(p))&=&\displaystyle \frac{\omega_n}{n} + W_0\,  \epsilon^2 + W\, \epsilon^4\\[4mm]
&&\displaystyle +v_0\, \left[\omega_n\, \left(1+\frac{n-1}2\,v_0\right)-\frac16\, \epsilon^2\, \int_{S^{n-1}}Ric_p(\Theta,\Theta) \, \right]\\[4mm]
&&\displaystyle +\frac{n-1}2\, \int_{S^{n-1}}\bar v^2\, -\frac16\, \epsilon^2\, \int_{S^{n-1}}\mathring{Ric}_p(\Theta,\Theta)\,\bar v \,  \\[4mm]
&&\displaystyle +\int_{S^{n-1}}\left(\mathcal{O}(\epsilon^5)+\epsilon^3\, L_p(v)+\epsilon^2\, Q_p^{(2)}(v)+Q_p^{(3)}(v)\right) 
\end{array}
$$
where
\[
\mathring Ric = Ric - \frac{R}{n}\, g
\] 
is the traceless Ricci curvature.
In order to compute the expansion of  $$\dot v_0:=\partial_{\bar v}v_0(p,\epsilon,0)(\bar v)=\left.\frac{\textnormal{d}}{\textnormal{d}s}\right|_{s=0}v_0(p,\epsilon,s\bar v).$$
we derivate with respect to $s$, at $s=0$, equality 
$$\textnormal{Vol}_g\, B^g_{\epsilon(1+v_0(p,\epsilon,s\bar v)+s\bar v)}(p)=\textnormal{Vol}\,B_\epsilon$$ using the expansion above.
Recall that we know $v_0(p,\epsilon,0)=\mathcal{O}(\epsilon^2)$. We find
$$
(1+\mathcal{O}(\epsilon^2))\, \omega_n\, \dot v_0=\frac 16\, \epsilon^2\, \int_{S^{n-1}} \mathring{Ric}_p(\Theta,\Theta)\, \bar v\, 
+\int_{S^{n-1}}\left(\mathcal{O}(\epsilon^5)+\epsilon^3\, L_p(\bar v)\right)  
$$
Finally 
$$
\dot v_0=\frac 1{6\, \omega_n}\, \epsilon^2\int_{S^{n-1}}\mathring{Ric}_p(\Theta,\Theta)\, \bar v\,  
+\int_{S^{n-1}}\left(\mathcal{O}(\epsilon^5)+\epsilon^3\,L_p(\bar v)\right)  
$$
This completes the proof of the Lemma. 
\end{proof}

We are now able to prove Lemma \ref{exp3}. 

\begin{proof} {\it (Lemma \ref{exp3})}.
We make a development up to power 2 in $\epsilon$,
of the function 
\[
\left.\frac{\textnormal{d}}{\textnormal{d}s}\right|_{s=0}\hat{\lambda}(p,\epsilon,s\bar v)  \, .
\]
From Proposition \ref{lambda}, we have
$$
\left.\frac{\textnormal{d}}{\textnormal{d}s}\right|_{s=0}\hat{\lambda}(p,\epsilon,s\bar v)=-\int_{\partial B_1}\hat g(V,\hat\nu)(\hat g (\hat \nabla \hat \phi,\hat \nu))^2 \mbox{dvol}_{{\hat g}_{\textnormal{in}}}
$$
where the deformation in a neighborhood of $\partial B_1$ is given by
\[
h(s,y) = \left( 1 + v_0(\epsilon,s\, \bar v) + s\, \bar v\right)y
\]
and
$$V(y)= \left.\frac{\partial h}{\partial s} \right|_{s=0} = \left[\partial_{\bar v}v_0(p,\epsilon,0)(\bar v)+\bar v\left(\frac{y}{|y|}\right)\right]y.$$
In that formula, the term $\hat g (\hat \nabla \hat \phi,\hat \nu)$ is computed with $s=0$ or equivalently $\bar v=0$.
From the definition of $\hat g$ and the expansion of the metric $g$, when $\bar v=0$ we have
\[
\begin{array}{lll}
\hat g_{ij} & = & \displaystyle \left(1+v_0(\epsilon,0)\right)^2\, \left(\delta_{ij}-\frac13\, \epsilon^2\, R_{ikjl}\, y^k\, y^l+\mathcal{O}(\epsilon^3)\right)\\[4mm]
\hat \nu & = & \displaystyle \left(1+v_0(\epsilon,0)\right)^{-1}\, \partial_r=\left(1+v_0(\epsilon,0)\right)^{-1}\, \frac{y}{|y|}\\
\end{array}
\]
The expansion of $\hat \phi(p, \epsilon, 0)$ is almost known: it suffices to replace $\epsilon$ by $\epsilon(1+v_0)$ in formula (\ref{phigeod}). We have
\[
\hat{\phi} = \phi_1+\epsilon^2f_2+\mathcal{O}(\epsilon^3)
\]
where
$$
f_2(y)=\left[R_{ij}\,y^i\,y^j-\frac{R_p}{n} |y|^2\right]\frac{\phi_1}{12}+R_p\;G_2(|y|).
$$
Using the notation
$ R^j\;_{k}\;^m\;_l=g^{ja}g^{mb}R_{akbl}$
we thus have on $\partial B_1$
\[
\begin{array}{lll}
\hat g (\hat \nabla \hat \phi,\hat \nu) & = & \displaystyle \left(1+v_0(\epsilon,0)\right)^{-1}\, \left(\delta_{ij}-\frac 13\, \epsilon^2\, R_{ikjl}\,y^k\,y^l\right)\, \cdot \\[4mm]
&&\displaystyle \cdot\, y^i\,\left(\delta^{jp}+\frac 13\, \epsilon^2\, R^j\;_{k}\;^m\;_l\, y^k\, y^l\right)\left(\frac\partial{\partial y^m}\phi_1+\epsilon^2\, \frac\partial{\partial y^m} f_2\right)+\mathcal{O}(\epsilon^3)\\[4mm]
& = &(1-v_0(\epsilon,0))^{-1}\, c_1+\epsilon^2\, \partial_r f_2+\mathcal{O}(\epsilon^3)\\
\end{array}
\]
where on the boundary 
$$
\partial_rf_2(y)=\frac{c_1} {12}\left[R_{ij}\,y^i\,y^j-\frac{R_p}{n}\right]+R_p\, G_2'(1) \, .
$$
Now we have to expand the measure on the boundary. This is classical and can be done directly from expansion (\ref{expmetric}). We have
$$
\left. \mbox{dvol}_{{\hat g}_{\textnormal{in}}}\right|_{\partial B_1}=(1+v_0)^{n}\left[1-\frac16\, Ric_p(\Theta,\Theta)\, \epsilon^2+\mathcal{O}(\epsilon^2)\right]\, \mbox{dvol}|_{S^{n-1}}
$$
where $\Theta$ is the vector in $T_pM$ whose coordinates are $y \in S^{n-1}$ and $\mbox{dvol}|_{S^{n-1}}$ is the Euclidean volume element induced on $S^{n-1}$.
For the term $\partial_{\bar v}v_0(p,\epsilon,0)(\bar v)$ appearing in $V$ we use Lemma \ref{v0punto}. 
We have
$$
\partial_{\bar v}v_0(p,\epsilon,0)(\bar v)=\frac{\epsilon^2}{6\, \omega_n}\, \int_{S^{n-1}}\mathring{Ric}_p(\Theta,\Theta)\, \bar v\, 
+\int_{S^{n-1}}[\mathcal{O}(\epsilon^5)+\epsilon^3\, L_p(\bar v)]\, 
$$
We finally obtain
$$
\left.\frac{\textnormal{d}}{\textnormal{d}s}\right|_{s=0}\hat{\lambda}(p,\epsilon,s\bar v)=C\, \epsilon^2\int_{S^{n-1}}\mathring{Ric}_p(\Theta,\Theta)\,\bar v\,
+\int_{S^{n-1}}[\mathcal{O}(\epsilon^5)+\epsilon^3\,L_p(\bar v)]\, 
$$
where
$$
C =-\frac{c_1^2}{6}-\frac{2\, c_1^2}{12}+\frac{c_1^2}{6}=-\frac{c_1^2}{6}.
$$
The proof of the Lemma follows at once.
\end{proof}

Summarizing the results of Lemmas \ref{exp1}, \ref{exp2} and \ref{exp3} we obtain the following:

\begin{Proposition}\label{lambdadl1}
Let $p \in M$, let $\epsilon$ and $\bar v$ be small enough. Then:
\[
\begin{array}{lll}
\hat \lambda (p,\epsilon,\bar v) & = & \displaystyle \lambda_1 + \hat \Lambda_0\, \epsilon^2 + \hat \Lambda\, \epsilon^4 \\[4mm]
&   &
\displaystyle - c_1\, \int_{S^{n-1}}\bar{v}\,H(\bar v)\, - \, \frac{c_1^2}{6}\, \epsilon^2\int_{S^{n-1}}\mathring{Ric}_p(\Theta,\Theta)\,\bar v\, \\[4mm]
&&\displaystyle +\int_{S^{n-1}}\left[\mathcal{O}(\epsilon^5)+\epsilon^3\,L_p(\bar v)+\epsilon^2\,Q_p^{(2)}(\bar v)+Q_p^{(3)}(\bar v)\right]\, 
\end{array}
\]
where $\Theta$ is the vector of $T_pM$ whose coordinates are $x \in S^{n-1}$ according with (\ref{nottheta}), and we agree with the convention about $L_p(v)$, $Q_p^{(2)}(v)$ and $Q_p^{(3)}(v)$ we gave before.
\end{Proposition}

\begin{proof} It suffices to put together the results of Lemmas \ref{exp1}, \ref{exp2} and \ref{exp3}.
\end{proof}

\section{Localisation of the obtained extremal domains}\label{localisationD}

Now we want to find the expansion of the function $\Psi_\epsilon(p)$ in power of $\epsilon$. Recall that
\[
\Psi_\epsilon(p) = \hat \lambda (p,\epsilon,\bar v(p,\epsilon))
\]
In order to find such expansion we will relate the first term in the expansion of $\bar v(p, \epsilon)$ to the curvature of the manifold at $p$. 

\medskip

The first term of the expansion of $\bar v(p, \epsilon)$ is related to the traceless Ricci curvature at $p$, as stated by the following:

\begin{Proposition}\label{propvbar} We have
\[ 
\bar v(p, \epsilon) =  -\frac{c_1}{12\, \alpha_2}\, \mathring Ric_p(\Theta, \Theta)\, \epsilon^2 + \mathcal{O}(\epsilon^3)= \frac{n}{12\, (\lambda_1 - n)}\, \mathring Ric_p(\Theta, \Theta)\, \epsilon^2 + \mathcal{O}(\epsilon^3)
\]
where $\Theta$ is the vector of $T_p M$ whose geodesic coordinates are $y \in S^{n-1}$, according with (\ref{nottheta}), and $\alpha_2$ is the eigenvalue of the operator $H$ defined in Proposition \ref{H} associated to the eigenspace $V_2$.
\end{Proposition}

\begin{proof} 
Let us recall that
$$
 F(p,\epsilon,\bar v(p, \epsilon))+\langle a(p, \epsilon),\cdot\rangle=0.
$$
where 
\[
\| \bar v (p, \epsilon) \|_{C^{2,\alpha}(S^{n-1})} + \|a(p, \epsilon)\| \leq c \, \epsilon^2
\]
Now, because $F(p,\epsilon,0)=\mathcal{O}(\epsilon^2)$ and because $\bar v(p, \epsilon)= \mathcal{O}(\epsilon^2)$, we can write
\[
\begin{array}{lll}
F (p,\epsilon,\bar v) & = & F (p,0,0)+\partial_\epsilon F (p,0,0)\, \epsilon\\[4mm]
&& \displaystyle + \partial_{\bar v}F (p,0,0)\, \bar v+\frac12\partial_\epsilon^2F (p,0,0)\,\epsilon^2+\mathcal{O}(\epsilon^3)\\[4mm]
&=&\displaystyle H(\bar v)+\frac12\, \partial_\epsilon^2F (p,0,0)\, \epsilon^2+\mathcal{O}(\epsilon^3)\\
\end{array}
\]
In the computation of the mixed derivatives of $\hat{\lambda}$ in the proof of Lemma \ref{exp3} we have already computed the expansion of $\hat g(\nabla \hat{\phi},\hat{\nu})$
for $\bar v=0$, so we directly deduce
\[
\begin{array}{lll}
F(p,\epsilon,0)&=& \displaystyle \epsilon^2\, \frac{c_1}{12}\, \left[R_{ij}(p)\,y^i\,y^j-\frac {R_p}{n}\right]+\mathcal{O}(\epsilon^3)\\[4mm]
&=&\displaystyle \epsilon^2\, \frac{c_1}{12}\, \mathring Ric_p(\Theta, \Theta)+\mathcal{O}(\epsilon^3).
\end{array}
\]
Then we have
\[
\partial_\epsilon^2F (p,0,0) = \frac{c_1}{6}\, \mathring Ric_p(\Theta, \Theta)
\]

\medskip

Writing
$$
a=a_p\, \epsilon^2+\mathcal{O}(\epsilon^3)
$$
and
 $$
\bar v=\bar v_p\, \epsilon^2+\mathcal{O}(\epsilon^3) \, 
$$
and considering the expansion of $F$, from equation (\ref{Fa}) we obtain
\begin{equation}
\label{HRic}
H(\bar v_p)+\frac{c_1}{12}\,\mathring Ric_p(\Theta, \Theta)=-\langle a_p,\cdot\rangle
\end{equation}
We know that $\bar v$, and hence $\bar v_p$, is $L^2$-orthogonal to $V_0\oplus V_1$ (see Propositions \ref{solmodnoyau}). Observe that $Ric(\Theta, \Theta)$ is $L^2(S^{n-1})$-orthogonal to $V_1$ since the function $\Theta \to Ric(\Theta, \Theta)$ is invariant when $\Theta$ is changed into $-\Theta$ and hence its $L^2$-projection over elements of the form $g(\Xi, \Theta)$ is 0 for every $\Xi$. Then $\mathring Ric(\Theta, \Theta)$ is $L^2(S^{n-1})$-orthogonal to $V_0\oplus V_1$. In fact $\mathring Ric(\Theta, \Theta)$ is the restriction on $S^{n-1}$ of a homogeneous polynomial of degree 2 which has mean 0, and then it is an eigenfunction for $-\Delta_{S^{n-1}}$ with eigenvalue $2n$. As $H$ preserves the eigenspaces of 
$-\Delta_{S^{n-1}}$ and his kernel is given by $V_1$ (see Proposition \ref{H}), we have that there exists a constant $\alpha_2\neq0$ such that
$$H\left(\mathring Ric(\Theta, \Theta)\right)=\alpha_2\,\mathring Ric(\Theta, \Theta)$$
From (\ref{HRic}) we obtain
$$
 -\langle a_p,\cdot\rangle = H\left(\bar v_p+\frac{c_1}{12\alpha_2}\, \mathring Ric(\Theta, \Theta)\right)
$$ 
i.e. $ \langle a_p,\cdot\rangle$ is in the image of $H$. But
it belongs also to the kernel of $H$, and then $a_p=0$ and 
\begin{equation}
\label{HRic2}
H\left(v_p+\frac{c_1}{12\alpha_2}\, \mathring Ric(\Theta, \Theta)\right)=0
\end{equation}
Now we remark that $\displaystyle \left(v_p+\frac{c_1}{12\alpha_2}\,\mathring Ric(\Theta, \Theta)\right)$ is orthogonal to $V_0\oplus V_1$, and then
\begin{equation}
\label{vRic2}
v_p=-\frac{c_1}{12\alpha_2}\,\mathring Ric(\Theta, \Theta)
\end{equation}

In order to complete the proof of the proposition we use equation (\ref{alpha2}) and Lemma \ref{lemma_c1} of the Appendix.
\end{proof}

Now we are able to give an expansion for the function $\Psi_\epsilon(p)$ in power of $\epsilon$. 

\begin{Proposition} 
We have:
\begin{equation}\label{rrr}
\displaystyle \Psi_\epsilon(p)  = 
 \lambda_1+\frac{\hat{\Lambda}_0}{R_p}\, \epsilon^2\, \left(R_p+{{\bf r}_p}\, \epsilon^2\right)+\mathcal{O}(\epsilon^5)
\end{equation}
where $\hat \Lambda_0$ is defined in Lemma \ref{exp1} (note that $\frac{\hat{\Lambda}_0}{R_p}$ is well defined also when $R_p=0$), and the function ${\bf r}$ can be written as
\[
{\bf r} = K_1\, \|Riem\|^2+K_2\, \|Ric\|^2+K_3\, R^2+K_4\, \Delta_g R
\]
for some constants $K_i$ only depending on $n$.
\end{Proposition}

\begin{proof}
Replacing $\bar v$ with its expansion given by Proposition \ref{propvbar} in the expansion of $\hat \lambda$ given by Proposition \ref{lambdadl1}, we obtain


\[
\begin{array}{lll}
\displaystyle \Psi_\epsilon(p)  &= &
\displaystyle \lambda_1+\hat{\Lambda}_0\, \epsilon^2+\hat{\Lambda}\, \epsilon^4
 -c_1\int_{S^{n-1}}\bar v\left(H(\bar v)+\frac{c_1}{6}\, \epsilon^2\, \mathring{Ric}_p(\Theta,\Theta)\right)
 +\mathcal{O}(\epsilon^5)\\[4mm]
 &= &
 \displaystyle \lambda_1+\hat{\Lambda}_0\, \epsilon^2+\hat{\Lambda}\, \epsilon^4
  -c_1\int_{S^{n-1}}\bar v\left(\alpha_2\, \bar v+\frac{c_1}{6}\, \epsilon^2\, \mathring{Ric}_p(\Theta,\Theta)\right)
   +\mathcal{O}(\epsilon^5)\\[4mm]
    &= &
    \displaystyle \lambda_1+\hat{\Lambda}_0\, \epsilon^2+\hat{\Lambda}\, \epsilon^4+
     \frac{c_1^3}{144\, \alpha_2}\, \epsilon^4\, \int_{S^{n-1}}(\mathring{Ric}_p(\Theta,\Theta))^2
      +\mathcal{O}(\epsilon^5)\\[4mm]
      &= &
          \displaystyle \lambda_1+\hat{\Lambda}_0\, \epsilon^2+\hat{\Lambda}\, \epsilon^4+
           \frac{c_1^3\, \omega_n}{72\, \alpha_2\, n(n+2)}\, \epsilon^4\, \left(\|Ric_p\|^2-\frac 1n \, R_p^2\right)
            +\mathcal{O}(\epsilon^5) \\[4mm]
        &= &
                 \displaystyle \lambda_1+\hat{\Lambda}_0\, \epsilon^2+\hat{\Lambda}\, \epsilon^4+
                  \frac{\lambda_1}{36(n+2)(n-\lambda_1)}\, \epsilon^4\, \left(\|Ric_p\|^2-\frac 1n \, R_p^2\right)
                   +\mathcal{O}(\epsilon^5)\\[4mm]
\end{array}
\]
where we used (\ref{vRic2}) from the second to the third line, the following two geometric formulas 
\begin{eqnarray*}
\int_{S^{n-1}} Ric(\Theta,\Theta) & = & \frac{\omega_n}{n}\, R_p \\
\int_{S^{n-1}} (Ric(\Theta,\Theta))^2 & = & \frac{\omega_n}{n(n+2)}\,  (2 \|Ric_p\|^2 + R_p^2) \, ,
\end{eqnarray*}
whose proofs can be found in \cite{Pac-Xu}, from the third to the fourth line, and the computation of $\alpha_2$ given in (\ref{alpha2}) and Lemma \ref{lemma_c1} to deduce the last line.
Define
$$
\begin{array}{lll}
{\bf r}_p&=&\displaystyle R_p\, \hat{\Lambda}_0^{-1}\, \left[\hat{\Lambda}+\frac{\lambda_1}{36(n+2)(n-\lambda_1)}\, \left(\|Ric_p\|^2-\frac 1n R_p^2\right)\right]\\[4mm]
&=&\displaystyle R_p\, \hat{\Lambda}_0^{-1}\, \left[\Lambda + \lambda_1\, \left( \frac{2\, W}{\omega_n} - \frac{R_p^2}{36\, n^2\, (n+2)} \right)
+\frac{\lambda_1}{36(n+2)(n-\lambda_1)}\,  \left(\|Ric_p\|^2-\frac 1n \, R_p^2\right)\right]\\
\end{array}
$$
Recalling the definition of $W$ and $\Lambda$ given in (\ref{W}) and (\ref{constL}), we obtain that
\[
{\bf r}_p = K_1\, \|Riem_p\|^2+K_2\, \|Ric_p\|^2+K_3\, R_p^2+K_4\, (\Delta_g R)_p
\]
where
\begin{equation}\label{constKi}
\begin{array}{lll}
K_1 & =  & \displaystyle \frac{1}{n\, (n+2) + 2\, \lambda_1}\, \left(18\, c^2 + \frac{\lambda_1}{10(n+4)}\right)\\[4mm]
K_2 & = &\displaystyle \frac{1}{n\, (n+2) + 2\, \lambda_1}\, \left(\frac{35}{3}\, c^2 + \frac{4\, \lambda_1}{15(n+4)}+\frac {n\lambda_1}{6(\lambda_1-n)}\right)\\[4mm]
K_3 & = & \displaystyle\frac{1}{n\, (n+2) + 2\, \lambda_1}\, \left(\frac{5n-3}{3n}\, c^2 - \frac{\lambda_1}{6(n+4)} + \frac{\lambda_1}{6n}-\frac {\lambda_1}{6(\lambda_1-n)}\right)\\[4mm]
K_4 & =  &\displaystyle \frac{1}{n\, (n+2) + 2\, \lambda_1}\, \left(\frac{6}{5}\, c^2 + \frac{3\, \lambda_1}{5(n+4)}\right)
\end{array}
\end{equation}
and formula (\ref{rrr}) follows at once. The fact that the constants $K_i$ depend only on $n$ comes immediately from the computation of $c^2$ by Lemma \ref{lemma_c} in the Appendix:
$$
c^2=\frac{(n+2)\, [2\lambda_1 + n(n-4)]}{12\, \lambda_1\, \omega_n}
$$
This completes the proof of the proposition.
\end{proof}
\begin{Remark}
We remark that
$K_1>0$ in order to justify our discussion about critical point of $\|Riem\|$ for Einstein metrics in the introduction. 
\end{Remark}

Now recalling that
$$
\Phi(p,\epsilon)=R_p\, \hat{\Lambda}_0^{-1}\, \frac{\Psi_\epsilon(p)-\lambda_1}{\epsilon^2}=-\frac{6n(n+2)}{n(n+2)+2\lambda_1}\frac{\Psi_\epsilon(p)-\lambda_1}{\epsilon^2}
$$
the proof of the second and third part of Theorem \ref{maintheorem} follows at once.

\section{Appendix I : On the first eigenfunction in the unit Euclidean ball}

In this Appendix we state and prove some relations between the first eigenfunction and the first eigenvalue of the Dirichlet Laplacian on the unit ball.

\begin{Lemma}\label{lemma_c1}
Let
\[
c_1 = \phi_1'(1)
\]
where $x \to \phi_1(|x|)$ is the first eigenfunction of the Dirichlet Laplacian on the unit ball, normalized in order to have $L^2$-norm equal to 1. Then 
$$
c_1=-\sqrt{\frac{2\lambda_1}{\omega_n}}
$$
where $\lambda_1$ is the first eigenvalue of the Dirichlet Laplacian on the unit ball.
\end{Lemma}

\begin{proof}
Recall that $\phi_1$ is the solution of
$$
\phi_1''+\frac{n-1}r\phi_1'+\lambda_1\phi_1=0
$$
with normalization
 $$
 1=\int_{B_1}\phi_1^2(|x|)\, \textnormal{d}x=\omega_n\int_0^1(\phi_1)^2\, r^{n-1}\, \textnormal{d}r=-\frac{2\, \omega_n}{n}\, \int_0^1\phi_1\, \phi_1'\, r^{n}\, \textnormal{d}r
 $$
and
$$
\lambda_1=\int_{B_1}|\nabla \phi_1(|x|)|^2\, \textnormal{d}x=\omega_n\int_0^1(\phi_1')^2\, r^{n-1}\, \textnormal{d}r
$$
Now let us compute
\[
\begin{array}{lll}
(r^{n}(\phi_1')^2)' & = & n\, r^{n-1}\, (\phi_1')^2+2r^{n}\, \phi_1'\, \phi_1''\\[4mm]
&=&\displaystyle n\, r^{n-1}\,(\phi_1')^2-2r^{n}\,\phi_1'\, \left(\frac{n-1}r\,\phi_1'+\lambda_1\, \phi_1\right)\\[4mm]
&=&(2-n)\, r^{n-1}\, (\phi_1')^2-2\lambda_1\, r^{n}\, \phi_1'\, \phi_1\\
\end{array}
\]
Integrating this relation between 0 and 1 we obtain
$$
c_1^2=\frac{2\lambda_1}{\omega_n}
$$
The proof of the Lemma follows at once, keeping in mind that $c_1$ is negative.
\end{proof}

\begin{Lemma}\label{lemma_c}
Let 
$$
c^2 =\frac{n+2}2\int_0^1\phi_1^2\,r^{n+1} \, \textnormal{d}r
$$
where $x \to \phi_1(|x|)$ is the first eigenfunction of the Dirichlet Laplacian on the unit ball, normalized in order to have $L^2$-norm equal to 1. Then 
$$
c^2=\frac{(n+2)\, [2\lambda_1 + n(n-4)]}{12\, \lambda_1\, \omega_n}
$$
where $\lambda_1$ is the first eigenvalue of the Dirichelt Laplacian on the unit ball.
\end{Lemma}

\begin{proof}
We have
$$
c^2=\frac{n+2}2\int_0^1\phi_1^2\,r^{n+1} \, \textnormal{d}r = -\int_0^1\phi_1\, \phi_1'\, r^{n+2}\, \textnormal{d}r
$$
Recall also that 
$$
\phi_1''+\frac{n-1}r\, \phi_1'+\lambda_1\phi_1=0
$$
with $\phi_1(1) = 0$, and $\phi_1$ is normalized by 
$$
(\omega_n)^{-1}=\int_0^1\phi_1^2\, r^{n-1}\,  \textnormal{d}r =-\frac{2}{n}\int_0^1\phi_1\, \phi_1'\, r^{n}\,  \textnormal{d}r
$$
We first compute
\[
\begin{array}{lll}
\displaystyle (r^{n+2}(\phi_1')^2)' & = &\displaystyle (n+2)\,r^{n+1}\,(\phi_1')^2+2r^{n+2}\,\phi_1'\,\phi_1''\\[4mm]
&=&\displaystyle (n+2)\,r^{n+1}\,(\phi_1')^2-2r^{n+2}\,\phi_1'\,\left(\frac{n-1}r\,\phi_1'+\lambda_1\,\phi_1\right)\\[4mm]
&=&\displaystyle (4-n)\,r^{n+1}\,(\phi_1')^2-2\lambda_1\,r^{n+2}\,\phi_1'\,\phi_1\\
\end{array}
\]
Integrating this relation between $0$ and $1$ we find
$$
c_1^2=(4-n)\int_0^1r^{n+1}\,(\phi_1')^2+2\lambda_1 c^2
$$
where $c_1 = \phi_1'(1)$.
We now compute
\[
\begin{array}{lll}
\displaystyle (r^{n+1}\,\phi_1\,\phi_1')' & = & \displaystyle(n+1)\,r^{n}\,\phi_1\,\phi_1'+r^{n+1}\,(\phi_1')^2+r^{n+1}\,\phi_1\,\phi_1''\\[4mm]
&=&\displaystyle(n+1)\,r^{n}\,\phi_1\,\phi_1'+r^{n+1}\,(\phi_1')^2-r^{n+1}\,\phi_1\,\left(\frac{n-1}r\,\phi_1'+\lambda_1\,\phi_1\right)\\[4mm]
&=&\displaystyle 2r^{n}\,\phi_1\,\phi_1'+r^{n+1}\,(\phi_1')^2-\lambda_1\,r^{n+1}\,\phi_1^2\\[4mm]
&=& \displaystyle (r^{n}\,\phi_1^2)'-n\,r^{n-1}\,\phi_1^2+r^{n+1}\,(\phi_1')^2-\lambda_1\,r^{n+1}\,\phi_1^2\\
\end{array}
\]
Integrating this relation between $0$ and $1$ we find
$$
0=-n\,(\omega_n)^{-1}+\int_0^1r^{n+1}\,(\phi_1')^2-\lambda_1\frac2{n+2}c^2 \, .
$$
We thus have at the end
$$
c^2=\frac{n+2}{12\lambda_1}\,\left[c_1^2+\frac{n(n-4)}{\omega_n}\right] \, .
$$
The proof of the Lemma follows at once from Lemma \ref{lemma_c1}.
\end{proof}

\section{Appendix II: The second eigenvalue of the operator $H$}
Here we compute the eigenvalue $\alpha_2$ of the operator $H$ associated to the eigenspace $V_2$. 
When $w$ is an homogeneous polynomial harmonic of degree 2 (abusively identified with its restriction
to the unit sphere) we have $\Delta_{S^{n-1}} w =-\mu_2 w=-2n\, w$ and $H(w)=\alpha_2\, w$.
We recall that $$H(w)=\left.(\partial_r\psi)\right|_{\partial B_1}+c_2\,w=\left.(\partial_r\psi)\right|_{\partial B_1}-(n-1)\,c_1\,w$$
 where $\psi$ is the solution of
 \[
 \left\{
\begin{array}{rcll}
	\displaystyle \Delta \psi + \lambda_{1}Ê\,Ê \psi & = & 0  & \textnormal{in} \qquad B_1 \\[3mm]
	\psi & = &  - \displaystyle  {c_1} \, w  & \textnormal{on}\qquad \partial B_1 \, 
\end{array}
\right.
\]
which is $L^2(B_1)$-orthogonal to $\phi_1$.
Decomposing $\psi$ in spherical harmonics, we see that $\psi(r,\theta)=b_2(r)\, w(\theta)$ where $b_2$ is the solution defined at $0$ of 
$$
\left\{
\begin{array}{lll}
r^2\,b''+(n-1)\,r\,b'+(r^2\,\lambda_1-2n)\,b=0& \mbox{in} &(0,1)\\[2mm]
b(1)=-c_1=-\phi_1'(1)&&
\end{array}
\right.
$$
From the definition of $H$, we see that
$$
\alpha_2=b_2'(1)+\phi_1''(1)=b_2'(1)+c_2=b_2'(1)-(n-1)\,c_1
$$
so we have to compute $b_2'(1)$.
Let us  verify that
$$
b_2(r)=-\left(\frac{\lambda_1}n\,\phi_1+\frac 1r \, \phi_1'\right)
$$
is the desired solution. Recall that
$$
\phi_1''+\frac{n-1}r\,\phi_1'+\lambda_1\phi_1=0,
$$
thus
$$
(\phi_1')''+\frac{n-1}r\,(\phi_1')'+\lambda_1\phi_1'=\frac{n-1}{r^2}\,\phi_1'
$$
Now 
$$b_2'=-\left(\frac{\lambda_1} n\, \phi_1'+\frac 1r\,\phi_1''\right)+\frac{ 1}{r^2}\,\phi'_1$$
and
$$b_2''=-\left(\frac{\lambda_1} n\,\phi_1''+\frac 1r\,\phi_1'''\right)+\frac{2}{r^2}\, \phi''_1-\frac{2}{r^3}\, \phi'_1$$
so 
$$
\begin{array}{lll}
\displaystyle b_2''+\frac{n-1}r\, b_2'+\lambda_1\, b_2&=&\displaystyle -\frac{1}{r}\,\frac{n-1}{r^2}\,\phi_1'+\frac{n-1}r\, \frac{ 1}{r^2}\, \phi'_1+2\, \frac{ 1}{r^2}\, \phi''_1-2\, \frac{1}{r^3}\, \phi'_1\\[4mm]
&=&\displaystyle -2\, \frac{ 1}{r^2}\, \left (\frac{n-1}r\, \phi_1'+\lambda_1\, \phi_1\right)-2\, \frac{1}{r^3}\, \phi'_1\\[4mm]
&=&\displaystyle -\frac{2n}{r^2}\, \left(\frac{\lambda_1}{n}\, \phi_1+\frac 1r\, \phi_1'\right)\\[4mm]
&=&\displaystyle \frac{2n}{r^2}\, b_2
\end{array}
$$
And of course $b_2(1)=-c_1$, so this is the desired solution.
Finally we have
\[
b_2'(1)=\frac{n^2-\lambda_1}n c_1
\]
and 
\begin{equation}\label{alpha2}
\alpha_2=\frac{n-\lambda_1}n c_1=\frac{\lambda_1-n}n\sqrt{\frac{2\lambda_1}{\omega_n}}\geq 0.
\end{equation}

\section{Appendix III : Differentiating with respect to the domain} 

In this Appendix we recall a useful result that allows to derivate the integral of a function with respect to a parameter $t$ that appears in the function and also in the domain of integration. The proof of the such a result can be found in \cite{Henry}, page 14. 

\begin{Proposition}\label{prop_henry}
Let $\Omega$ a smooth bounded domain of $\mathbb{R}^n$ and
\[
h : (-r,r) \times \Omega \to \mathbb{R}^n
\]
a smooth function, where $r$ is a positive constant, such that $h(0, p) = p$ for all $p \in \Omega$. Let 
\[
f  : \mathbb{R} \times \mathbb{R}^n \to \mathbb{R}
\]
a smooth function. Let $\Omega_t = h(t, \Omega_0)$, $V(t, h(t,p)) = \frac{\partial h}{\partial t}(t,p)$ and $N(t,q)$ the unit outward normal at $q \in \partial\Omega_t$. Then
\begin{equation}\label{henry1}
\frac{\partial}{\partial t}  \int_{\Omega_t} f \,  =  \int_{\Omega_t} \frac{\partial f}{\partial t}\, \mbox{d}x  +  \int_{\partial \Omega_t} f \, \langle V, N \rangle\, \mbox{d}s
\end{equation}
and 
\begin{equation}\label{henry2}
\frac{\partial}{\partial t}  \int_{\partial \Omega_t} f \, \mbox{d}s  =  \int_{\partial \Omega_t} \left(\frac{\partial f}{\partial t} + \langle V, N \rangle\, \langle \nabla_x f, N \rangle + H\, \langle V, N \rangle\,f\right)  \mbox{d}s 
\end{equation}
where $\langle \cdot, \cdot \rangle$ denote the scalar product in $\mathbb{R}^n$, $s$ denote the area element of $\partial \Omega_t$ and $H$ is the mean curvature of $\partial \Omega_t$.
\end{Proposition}

\end{document}